\documentclass[10pt]{amsart}

\usepackage{graphicx,amssymb,latexsym,amsfonts,txfonts,amsmath,amsthm}
\usepackage{pdfsync,color}

\usepackage{hyperref}
\hypersetup{
    colorlinks=true,       
    linkcolor=blue,          
    citecolor=blue,        
    filecolor=blue,      
    urlcolor=blue           
}

\advance\hoffset by -0.9 truecm
\newcommand{\s}{\vspace{0.3cm}}

\newtheorem{theo}{Theorem}
\newtheorem{coro}{Corollary}
\newtheorem{prop}{Proposition}
\newtheorem{lemm}{Lemma}

\theoremstyle{remark}
\newtheorem{rema}{\bf Remark}

\begin{document}  

\title[Quasiplatonic curves with symmetry group ${\mathbb Z}_{2}^{2} \rtimes {\mathbb Z}_{m}$]{Quasiplatonic curves with symmetry group ${\mathbb Z}_{2}^{2} \rtimes {\mathbb Z}_{m}$ \\are definable over ${\mathbb Q}$}

\author{Rub\'en A. Hidalgo}
\email{ruben.hidalgo@ufrontera.cl}
\address{Departamento de Matem\'atica y Estad\'{\i}stica, Universidad de La Frontera, Temuco, Chile.}

\author{Leslie Jim\'enez}
\email{leslie.jimenez@liu.se}
\address{Matematiska Institutionen, Link\"{o}pings Universitet, Link\"{o}ping, Sweden}

\author{Sa\'ul Quispe}
\email{saul.quispe@ufrontera.cl}
\address{Departamento de Matem\'atica y Estad\'{\i}stica, Universidad de La Frontera, Temuco, Chile.}

\author{Sebasti\'an Reyes-Carocca}
\email{sebastian.reyes@ufrontera.cl}
\address{Departamento de Matem\'atica y Estad\'{\i}stica, Universidad de La Frontera, Temuco, Chile.}

\thanks{Partially supported by Fondecyt Project 1150003, Postdoctoral Fondecyt Projects 3160002 and 3140050, Beca Chile Fellowship for Postdoctoral Research and Project Anillo ACT 1415 PIA-CONICYT}

\subjclass[2000]{14H37, 14H55, 14H25, 14H30, 30F10}
\keywords{Riemann surfaces, regular Belyi pairs, dessins d'enfants, Galois action, Jacobian variety}

\begin{abstract}
It is well known that every closed Riemann surface $S$ of genus $g \geq 2$, admitting a group $G$ of conformal automorphisms so that $S/G$ has triangular signature, can be defined over a finite extension of ${\mathbb Q}$. It is interesting to know, in terms of the algebraic structure of $G$, if $S$ can in fact be defined over ${\mathbb Q}$. This is the situation if $G$ is either abelian or isomorphic to $A \rtimes {\mathbb Z}_{2}$, where $A$ is an abelian group. On the other hand, as shown by Streit and Wolfart, if $G \cong {\mathbb Z}_{p} \rtimes {\mathbb Z}_{q}$ where $p,q>3$ are prime integers, then $S$ is not necessarily definable over ${\mathbb Q}$. In this paper, we observe that if $G\cong{\mathbb Z}_{2}^{2} \rtimes {\mathbb Z}_{m}$ with $m \geq 3$, then $S$ can be defined over ${\mathbb Q}$. Moreover, we describe explicit models for $S$, the corresponding groups of automorphisms and an isogenous decomposition of their Jacobian varieties as product of Jacobians of hyperelliptic Riemann surfaces.

\end{abstract}

\maketitle

\section{Introduction}
As it was defined by Grothendick in \cite{Grothendieck}, a dessin d'enfant of genus $g$ is a bipartite map on a closed orientable surface of genus $g$. The dessin d'enfant induces a unique, up to isomorphism, Riemann surface structure $S$ together with a non-constant meromorphic map $\beta:S \to \widehat{\mathbb C}$ whose branch values are contained in the set $\{\infty,0,1\}$; $S$ is called a Belyi curve, $\beta$ a Belyi map and $(S,\beta)$ a Belyi pair. Conversely, as a consequence of the uniformization theorem, each Belyi pair $(S,\beta)$ induces a dessin d'enfant (the bipartite map is provided by the preimage under $\beta$ of the closed interval $[0,1]$). 

A Belyi pair $(S,\beta)$ (and the corresponding dessin d'enfant) is called regular (in which case $S$ is called a quasiplatonic curve) if $\beta$ is a regular branched cover, that is, if there is a group of conformal automorphisms of $S$ being the deck group of $\beta$ (see, for instance, \cite{CJSW, Wolfart} for more details). It is well known that a finite group $G$ can be seen as the deck group of a Belyi pair (we say that the action has triangular signature) if and only if it can be generated by two elements \cite{Wolfart}. 

Two Belyi pairs $(S_{1},\beta_{1})$ and $(S_{2},\beta_{2})$ are called isomorphic if there is an isomorphism (a biholomorphic map) $h:S_{1} \to S_{2}$ such that $\beta_{1}=\beta_{2} \circ h$. Let us note that if $(S_{1},\beta_{1})$ and $(S_{2},\beta_{2})$ are regular Belyi pairs with respective deck groups $G_{1}$ and $G_{2}$, then the isomorphism $h$ conjugates $G_{1}$ onto $G_{2}$.

As a consequence of Belyi's theorem \cite{Belyi}, each Belyi pair can be defined over the field of algebraic numbers $\overline{\mathbb Q},$  that is, there is an isomorphic Belyi pair $(C,\beta)$, where $C$ (as an irreducible algebraic curve) and $\beta$ (as a rational map) are defined over $\overline{\mathbb Q}$. This fact permits to define an action of the absolute Galois group ${\rm Gal}(\overline{\mathbb Q}/{\mathbb Q})$ on Belyi pairs (or dessins d'enfants), as follows. Let $P_{1},\ldots,P_{r}$ be polynomials (with coefficients in $\overline{\mathbb Q}$) defining $C$, that is, $C$ is the set of common zeroes of these polynomials. Each $\sigma \in {\rm Gal}(\overline{\mathbb Q}/{\mathbb Q})$ provides new polynomials $P_{1}^{\sigma},\ldots,P_{r}^{\sigma}$ (where $P_{j}^{\sigma}$ is obtained from $P_{j}$ by applying $\sigma$ to its coefficients).
These new polynomials define a new irreducible algebraic curve $C^{\sigma}$. Similarly, we may apply $\sigma$ to the coefficients of $\beta$ and at the end we obtain a new Belyi pair $(C^{\sigma},\beta^{\sigma})$. It is well known that the absolute Galois group acts faithfully. Recently, Gonz\'alez-Diez and Jaikin-Zapirain \cite{Gabino-Andrei} proved that the absolute group acts faithfully on regular Bely pairs (even at the level of quasiplatonic curves).

The fixed points of the absolute group action are provided by those Belyi pairs which can be defined over ${\mathbb Q}$. It is a difficult task to decide if a given Belyi pair (or dessin d'enfant) can or cannot be definable over ${\mathbb Q}$. In the case of regular ones, some answers are known in terms of the corresponding deck group $G$. For instance, if 
either $G$ is an abelian group or a semidirect product $A \rtimes {\mathbb Z}_{2}$, where $A$ is abelian group, then the corresponding regular Belyi pair can be defined over ${\mathbb Q}$ (see \cite{H1} and \cite{H2}). On the other hand, in \cite{SW} it was noted that if $G= \langle a,b: a^p = b^q = 1, bab^{-1} = a^{n}\rangle \cong {\mathbb Z}_{p} \rtimes {\mathbb Z}_{q}$, where $p,q>3$ are prime integers and $n^{q} \equiv 1 \mod{p}$, then the regular Belyi pair is not necessarily definable over ${\mathbb Q}$.

In this paper we consider regular Belyi pairs $(S,\beta)$ with deck group $G \cong {\mathbb Z}_{2}^{2} \rtimes {\mathbb Z}_{m}$ where $m \geq 2$. As previously noted, for the case $m=2$ these are definable over ${\mathbb Q}$. So we only need to take care of the case $m \geq 3$. Also, as the abelian situation is also definable over ${\mathbb Q}$, we assume $G$ to be non-abelian. Theorem \ref{main} asserts that for these left cases such regular Belyi pairs can be defined over ${\mathbb Q}$. 

We are also able to construct explicit rational models of these pairs, their full groups of conformal automorphisms and isogenous decompositions of their Jacobian varieties as a product of Jacobian varieties of hyperelliptic Riemann surfaces.

\s
\noindent{\bf{Acknowledgments.}} The authors are very grateful to Professor Anita Rojas for sharing her MAGMA routines with us; they were very useful for the calculations in the last section of this paper. 


\section{Main results}
We consider regular Belyi pairs $(S,G)$, where $G\cong {\mathbb Z}_{2}^{2} \rtimes {\mathbb Z}_{m}$,  $m \geq 3$, and $G$ non-abelian. 

\subsection{Signatures}
Before we proceed to our main result, we first describe the possible signatures for the quotient orbifold $S/G$.

\begin{prop}
Let $(S,\beta)$ be a regular Belyi pair of genus $g \geq 2$ admitting $G  \cong {\mathbb Z}_{2}^{2} \rtimes {\mathbb Z}_{m}$, where $m \geq 3$ and $G$ non-abelian, as its deck group. Then the possible signatures for the quotient $S/G$ are: 
\begin{enumerate}
\item $(0;2,2q,4q)$ if $m=2q$ and $q \geq 3$ is odd, or 
\item $(0;2,m,m)$ if $m \geq 6$ is either divisible by $3$ or by $4$.
\end{enumerate}
\end{prop}
\begin{proof}
It is not difficult to see that we only have two cases to consider (up to automorphisms). 
\begin{enumerate}
\item[(i)] $G=\langle a,b,t: a^{2}=b^{2}=(ab)^{2}=t^{m}=1, tat^{-1}=a, tbt^{-1}=ab\rangle$,
\item[(ii)] $G=\langle a,b,t: a^{2}=b^{2}=(ab)^{2}=t^{m}=1, tat^{-1}=b, tbt^{-1}=ab\rangle$.
\end{enumerate}

In case (i), as $tat^{-1}=a, tbt^{-1}=ab$ and $tabt^{-1}=b$, the integer $m$ must be even. Moreover, in this situation
$$G=\langle b,t: b^{2}=t^{m}=[t,b]^{2}=1, (tb)^{2}=(bt)^{2}\rangle,$$
where $[x,y]=xyx^{-1}y^{-1}$.
So it can be seen that $bt$ has order $m$ if $m$ is divisible by $4$ and order $2m$ otherwise; in particular, $S/G$ has  signature $(0;2,m,m)$ if $m$ is divisible by $4$ and $(0;2,m,2m)$ if $m=2q$ with $q \ge 3$ odd.

In case (ii), as $tat^{-1}=b$, $tbt^{-1}=ab$ and $tabt^{-1}=a$, the integer $m$ must be divisible by $3$. Moreover, in this situation
$$G=\langle a,t: a^{2}=t^{m}=[a,t]^{2}=1, t^{3}=(at)^{3}\rangle,$$
so it can be seen that $at$ has order $m$ and that $S/G$ has signature $(0;2,m,m)$.
\end{proof}

%

\subsection{Main theorem}
Our main result provides the explicit algebraic descriptions of $S$, its full group of conformal automorphisms ${\rm Aut}(S)$ and an isogenous  decomposition of its Jacobian variety $JS$ as product of  Jacobian varieties of hyperelliptic curves.

\s
\begin{theo}\label{main} Let $m \ge 3$ and $\omega_m=\mbox{exp}(2 \pi i /m).$ Let $(S,\beta)$ be a regular Belyi pair of genus $g \geq 2$ admitting a non-abelian semidirect product $G =\langle a,b\rangle \rtimes \langle t \rangle \cong {\mathbb Z}_{2}^{2} \rtimes {\mathbb Z}_{m}$ as its deck group. Then $(S,\beta)$ is definable over ${\mathbb Q}$ and the following holds.

\begin{enumerate}
\item \label{caso1} If $S/G$ has signature $(0;2,2q,4q)$, where $m=2q$ and $q \geq 3$ is odd, then 
$g=2(q-1)$ and the Belyi pair $(S,\beta)$ is unique up to isomorphisms. Moreover,
\begin{enumerate}
\item $S$ can be described by the algebraic curve
\begin{equation*}  \left\{\begin{array}{l}
y^{2}=x^{q}-1\\
z^{2}=x^{m}-1
\end{array}
\right\} \subset {\mathbb C}^{3},
\end{equation*}
the Belyi map corresponds to $\beta(x,y,z)=x^m$ and
$$a(x,y,z)=(x,-y,z), \; b(x,y,z)=(x,y,-z), \; t(x,y,z)=\left(\omega_{m} x, (iz)/y, z  \right).$$

\item The group $G$ is the full group of conformal automorphisms of $S$. 

\item The Jacobian variety $JS$ is isogenous to $(JS_{b})^{4}$, where 
$$S_{b}: y^{2}=x^{q}-1.$$
\end{enumerate}

\s

\item \label{caso2} If $S/G$ has signature $(0;2,m,m)$, where $m \geq 6$ is either divisible by $3$ or by $4$, then 
$g=m-3$ and the following holds. 

\begin{enumerate}
\item \label{caso2a} If $m=3l$, $l \geq 2$, is not divisible by $12$, then $(S,\beta)$ is unique up to isomorphisms. Moreover, 
\begin{enumerate}
\item $S$ is described by the algebraic curve
$$ \left\{\begin{array}{l}
y^{2}=(x^{l}-1)(x^{l}-\omega_{3}^{2})=x^{2l}+\omega_{3} x^{l}+\omega_{3}^{2}\\
z^{2}=(x^{l}-\omega_{3})(x^{l}-\omega_{3}^{2})=x^{2l}+x^{l}+1
\end{array}
\right\} \subset {\mathbb C}^{3},
$$
the Belyi map corresponds to $\beta(x,y,z)=x^m$ and 
$$a(x,y,z)=(x,-y,z), \; b(x,y,z)=(x,y,-z), \;
t(x,y,z)=\left(\omega_{m} x, -\omega_{3} z, \frac{\omega_{3} yz}{x^{l}-\omega_{3}^{2}}  \right).$$

\item The Riemann surface $S$ has the extra automorphism
$$u(x,y,z)=\left(\frac{1}{x}, \frac{\omega_{3}yz}{x^{l}(x^{l}-\omega_{3}^{2})}, \frac{z}{x^{l}} \right).$$

\item The group of conformal automorphisms ${\rm Aut}(S)$ is generated by $a,b,t$ and $u$. In fact,
$${\rm Aut}(S)=\langle t,u: u^{4}=t^{m}=(ut)^{2}=1, t^{3}=(u^{2}t)^{3}, ([t^{-1},u] u^{-1})^{2}=1\rangle,$$
has order $8m$ and $S/{\rm Aut}(S)$ has signature $(0;2,4,m)$. 

\item The Jacobian variety $JS$  is isogenous to $(JS_{a})^{3}$, where 
$$S_{a}: z^{2}=(x^{l}-\omega_{3})(x^{l}-\omega_{3}^{2})=x^{2l}+x^{l}+1.$$

In fact,
$$JS \sim (JS_{a,1})^{3} \times (JS_{a,2})^{3},$$
where
$$S_{a,1}: w_{1}^{2}=(1-v_{1})^{l}+2\sum_{j=0}^{l}\binom{2l}{2j}v_{1}^{j},$$
and 
$$S_{a,2}: w_{2}^{2}=v_{2}\left((1-v_{2})^{l}+2\sum_{j=0}^{l}\binom{2l}{2j}v_{2}^{j}\right).$$
\end{enumerate}

\item \label{caso2b} If $m=4l$, $l \geq 2$, is not divisible by $12$, then $(S,\beta)$ is unique up to isomorphisms. Moreover, 
\begin{enumerate}
\item $S$ is described by the algebraic curve
\begin{equation*}  \left\{\begin{array}{l}
y^{2}=x^{2l}-1\\
z^{2}=x^{m}-1
\end{array}
\right\} \subset {\mathbb C}^{3},
\end{equation*}the Belyi map corresponds to 
$\beta(x,y,z)=x^m$
and
$$a(x,y,z)=(x,-y,z), \; b(x,y,z)=(x,y,-z), \;
t(x,y,z)=\left(\omega_{m} x, \frac{iz}{y}, z  \right).$$

\item The Riemann surface $S$ has the extra automorphism
$$u(x,y,z)=\left(\frac{1}{x}, \frac{i y}{x^{l}}, \frac{iz}{x^{2l}} \right).$$

\item The group of conformal automorphisms ${\rm Aut}(S)$ is generated by $a,b,t$ and $u$. In fact, 
$${\rm Aut}(S)=\langle t,u: u^{4}=t^{m}=(tu)^{2}=[u^{2},t]^{2}=[u^{2},tut^{-1}]=1,(tu^{2})^{2}=(u^{2}t)^{2}\rangle$$
has order $8m$ and $S/{\rm Aut}(S)$ has signature $(0;2,4,m)$. 

\item The Jacobian variety $JS$ is isogenous to the product of $JS_{a,2} \times (JS_{b})^{3}$, where 
$$S_{a,2}: w_{2}^{2}=v_{2}(v_{2}^{2l}-1), \quad S_{b}: y^{2}=x^{2l}-1.$$

\end{enumerate}

\item \label{caso2c} If $m$ is divisible by $12$, then there are exactly two non-isomorphic pairs $(S,\beta)$; they are algebraically represented as in \eqref{caso2a} and \eqref{caso2b} above.

\item \label{caso4} In any of the above cases \eqref{caso2a}, \eqref{caso2b} and \eqref{caso2c}, the subgroup $\langle a,b \rangle$ is the unique subgroup (so a normal subgroup) of ${\rm Aut}(S)$ isomorphic to ${\mathbb Z}_{2}^{2}$, ${\rm Aut}(S)/\langle a,b \rangle \cong {\mathbb D}_{m}$,  and $S/\langle a,b \rangle$ is the Riemann sphere with exactly $m$ cone points, each one of order two.

\end{enumerate}
\end{enumerate}
\end{theo}

\section{Some remarks concerning Theorem \ref{main}}\label{observacion}

\subsection{Equations over ${\mathbb Q}$}
The provided curves in Theorem \ref{main}, with the only exception of the case \eqref{caso2a}, are defined over ${\mathbb Q}$. In the left case the provided curve is defined over a degree two extension of $\mathbb{Q};$ namely ${\mathbb Q}(\omega_{3}).$ However, the uniqueness property asserts that it is definable over ${\mathbb Q}$. In this case, we may follow the computational method presented in \cite{HR} to find a rational model (this is done in Section \ref{modelito}).
 
\subsection{Projective models} Projective models associated to the affine ones provided in Theorem \ref{main}, are the following:

\begin{enumerate}

\item[(1)] In case $(1),$ given by the curve
$$\hat{C}:\; \left\{\begin{array}{lll}
y^{2}w^{q-2} & = & x^{q}-w^{q}\\
z^{2}w^{2q-2} & = & x^{2q}-w^{2q}
\end{array}
\right\}
\subset {\mathbb P}^{3}
$$with Belyi map $\beta([x:y:z:w])=(x/w)^{m}$. Moreover, 

\begin{enumerate}
\item[(a)]  If $\beta(p) \in {\mathbb C}$ is not a $2q$-root of unity, then $\hat{C}$ is smooth at $p$.

\item[(b)] If $\beta(p) \in {\mathbb C}$ is a $q$-root of unity, then $p$ is a node of $\hat{C}$ (that is, locally looks like two cones glued at a common vertex).

\item[(c)] If $\beta(p) \in {\mathbb C}$ is a $2q$-root of unity but not a $q$-root of unity, then $p$ is a singular point of $\hat{C}$ of cone type (i.e., locally looks as a topological disc).

\item[(d)] Above $\infty$ we have two points on $\hat{C}$;  each one being a singular point of cone type (similarly as above). 

\end{enumerate}

In particular, a smooth model of $\hat{C}$ is obtained by separating at its $q$ nodes by a blowing-up process.

\item[(2a)] In case $(2a),$ given by the curve 
$$\hat{C}:\; \left\{\begin{array}{l}
y^{2}w^{2l-2}=(x^{l}-w^{l})(x^{l}-\omega_{3}^{2}w^{l})=x^{2l}+\omega_{3} x^{l}w^{l}+\omega_{3}^{2}w^{2l}\\
z^{2}w^{2l-2}=(x^{l}-\omega_{3}w^{l})(x^{l}-\omega_{3}^{2}w^{l})=x^{2l}+x^{l}w^{l}+w^{2l}
\end{array}
\right\} \subset {\mathbb P}^{3}
$$
and the Belyi map corresponds to $\beta([x:y:z:w])=(x/w)^{m}$. Moreover,

\begin{enumerate}
\item Observe that $\hat{C}$ is reducible as it contains a projective line $L$ at infinity. The surface $S$ corresponds to the union of $C$ with certain $4$ points of $L$.

\item If $\beta(p) \in {\mathbb C}$ is not a $3l$-root of unity, then $\hat{C}$ is smooth at $p$.

\item If $\beta(p) \in {\mathbb C}$ is a $l$-root of unity or a $l$-root of $\omega_{3}^{2}$, then $p$ is a node of $\hat{C}$ (that is, locally looks like two cones glued at a common vertex).

\item If $\beta(p) \in {\mathbb C}$ is a $l$-root of unity or a $l$-root of $\omega_{3}$, then $p$ is a singular point of $\hat{C}$ of cone type (i.e., locally  looks as 
a topological disc).

\end{enumerate}

The difference of the two equations of $C$ permits to obtain $x^{l}$:
$$(\omega_{3}-1)x^{l}-y^{2}w^{l-2}+z^{2}w^{l-2}-(2+\omega_{3})w^{l}=0$$

Now, using this value of $x^{l}$ in the first equation of $C$ permits to obtain $x^{2l}$:
$$(1-\omega_{3})x^{2l}-y^{2}w^{2l-2}+\omega_{3}z^{2}w^{2l-2}-\omega_{3}(1-\omega_{3})w^{2l}=0$$
and this first equation is the same as
$$y^{4}+z^{4}-2y^{2}z^{2}-3z^{2}w^{2}+3(1+\omega_{3})y^{2}w^{2}=0$$

In this way, the above affine curve can be written as follows:
$$C:\; \left\{\begin{array}{l}
y^{4}+z^{4}-2y^{2}z^{2}-3z^{2}+3(1+\omega_{3})y^{2}=0\\
(1-\omega_{3})x^{2l}-y^{2}+\omega_{3}z^{2}-\omega_{3}(1-\omega_{3})=0\\
(\omega_{3}-1)x^{l}-y^{2}+z^{2}-(2+\omega_{3})=0
\end{array}
\right\} \subset {\mathbb C}^{3}
$$
and its projectivization as
$$\hat{C}:\; \left\{\begin{array}{l}
y^{4}+z^{4}-2y^{2}z^{2}-3z^{2}w^{2}+3(1+r)y^{2}w^{2}=0\\
(1-\omega_{3})x^{2l}-y^{2}w^{2l-2}+\omega_{3}z^{2}w^{2l-2}-\omega_{3}(1-\omega_{3})w^{2l}=0\\
(\omega_{3}-1)x^{l}-y^{2}w^{l-2}+z^{2}w^{l-2}-(2+\omega_{3})w^{l}=0
\end{array}
\right\} \subset {\mathbb P}^{3}
$$
which is now irreducible. This has two singular points at infinity, given by the two points $[0:1:\pm 1:0]$, each one, after desingularization, produces two smooth points (i.e., we obtain the four points at infinity).

\item[(2b)] In case $(2b),$ given by the curve
$$\hat{C}:\; \left\{\begin{array}{lll}
y^{2}w^{2l-2} & = & x^{2l}-w^{2l}\\
z^{2}w^{4l-2} & = & x^{4l}-w^{4l}
\end{array}
\right\}
\subset {\mathbb P}^{3}
$$
with Belyi map $\beta([x:y:z:w])=(x/w)^{m}$. Moreover,

\begin{enumerate}
\item[(a)] If $\beta(p) \in {\mathbb C}$ is not a $4l$-root of unity, then $\hat{C}$ is smooth at $p$.

\item[(b)] If $\beta(p) \in {\mathbb C}$ is a $2l$-root of unity, then $p$ is a node of $\hat{C}$ (that is, locally looks like two cones glued at a common vertex).

\item[(c)] If $\beta(p) \in {\mathbb C}$ is a $4l$-root of unity but not a $2l$-root of unity, then $p$ is a singular point of $\hat{C}$ of cone type (i.e., locally looks like a topological disc).

\item[(d)] Above $\infty$ we have four points on $\hat{C}$;  each one being a smooth point. 
\end{enumerate}
\end{enumerate}

\subsection{Fiber product} 
In any of the cases in Theorem \ref{main}, the surface $S$ is just the fiber product of $(S_{a},\pi_{a}(x,z)=x$) and $(S_{b},\pi_{b}(x,y)=x)$.

\subsection{Hyperelliptic cases}
If we take $m=6$ in case (\ref{caso2a}), then $\langle t,u \rangle \cong {\mathbb Z}_{2} \times {\mathfrak S}_{4}$, where the ${\mathbb Z}_{2}$ component is generated by an element of order two with exactly $8$ fixed points (that is, the hyperelliptic involution). It follows that $S$ is the only hyperelliptic Riemann surface of genus three admitting as group of conformal automorphisms ${\mathfrak S}_{4}$. In fact, this is the only hyperelliptic situation appearing in Theorem \ref{main} (see Proposition \ref{hipereliptico} below).

\begin{prop}\label{hipereliptico}
The only hyperelliptic situation in Theorem \ref{main} is for $m=6$ in case (\ref{caso2a}).
\end{prop}
\begin{proof}
Let us consider the group $K=\langle a, b \rangle \cong {\mathbb Z}_{2}^{2}$ as above. By Theorem \ref{main}, we may identify $S/K$ with the Riemann sphere $\widehat{\mathbb C}$. Let us consider a regular branched covering $P:S \to \widehat{\mathbb C}$ with $K$ as its deck group. As the number of fixed points of $a$, $b$ and $ab$ is less than $2g+2$, it follows that $\iota \notin K$, in particular, there is an order two M\"obius transformation $\tau$ so that $P \circ \iota=\tau \circ P$. As $\iota$ cannot have a common fixed point with $a$, $b$ and $a b$, it follows that $\tau$ cannot fix any of the branch values of $P$. It follows that (as the fixed points of $\iota$ are projected by $P$ to the two fixed points of $\tau$) $2g+2 \leq 8$, i.e., $g \leq 3$. As the cases $g=1,2$ are not possible, we must have $g=3$ (this is only possible for $m=6$).
\end{proof}

\subsection{Completely decomposable Jacobians}\label{isogenia}
Isogenous decompositions of Jacobian varieties with group action (see Sections \ref{KR} and \ref{GAD}).  have been extensively studied from different points of view; see for example  \cite{CR}, \cite{yo}, \cite{K-R}, \cite{LR} and \cite{RJ}. In particular, completely decomposable Jacobians (having only elliptic factors) are a big subfield of study; see for example \cite{serre} and \cite{paulhus2}. In our case, in the Theorem \ref{main} we have
\begin{itemize}
\item In case (1) for $q=3$, $S$ has genus $g=4$ and its Jacobian variety is isogenous to $E_{1}^4$, where $E_{1}$ is the elliptic curve $y^{2}=x^{3}-1$. 

\item In case (2(a)i) for $l=2$, $S$ has genus $g=3$ and its Jacobian variety is isogenous to $E_{2}^3$, where $E_{2}$ is the elliptic curve $y^{2}=x(3x^{2}+10x+3)$ 
\end{itemize}

\subsection{Fuchsian uniformizations}
In each case as in Theorem \ref{main} one may provide the corresponding Fuchsian uniformizations. 

\subsubsection{Case (1): $m=2q$, $q \geq 3$ odd}
Let us consider the triangular group
$$\Delta=\langle x,y: x^{2m}=y^{m}=(xy)^{2}=1\rangle$$
and the surjective homomorphism 
$$\Theta:\Delta \to \langle t: t^{m}=1\rangle$$
$$\Theta(x)=t, \; \Theta(y)=t^{-1}.$$

The kernel of $\Theta$ is the subgroup
$$K=\langle \alpha_{1},\ldots,\alpha_{m+1}: \alpha_{1}^{2}=\cdots=\alpha_{m+1}^{2}=\alpha_{1} \cdots \alpha_{m+1}=1\rangle,$$
where $$\alpha_{j}=x^{j}yx^{1-j}, \quad j=1,\ldots,m$$
$$\alpha_{m+1}=x^{m}.$$

Let us now consider the surjective homomorphism
$$\eta:K \to \langle a,b:a^{2}=b^{2}=(ab)^{2}=1\rangle$$
where
$$\eta(\alpha_{j})=\left\{\begin{array}{ll}
b, & j \equiv 1 \mod 2, \; j \neq m+1\\
ab, & j \equiv 0 \mod 2\\
a, & j=m+1
\end{array}
\right. 
$$
and let $\Gamma$ be its kernel. Then $S={\mathbb H}^{2}/\Gamma$, $S/\langle a,b \rangle = {\mathbb H}^{2}/K$ and $S/G={\mathbb H}^{2}/\Delta$.

\subsubsection{Case (2): $m \in \{3l,4l\}$, $m \geq 4$}
Let us consider the triangular group
$$\Delta_{0}=\langle x,z: x^{m}=z^{2}=(xz)^{4}=1\rangle$$ 
the index two subgroup ($y=zxz$)
$$\Delta=\langle x,y: x^{m}=y^{m}=(xy)^{2}=1\rangle$$
and the surjective homomorphism 
$$\Theta:\Delta \to \langle t: t^{m}=1\rangle$$
$$\Theta(x)=t, \; \Theta(y)=t^{-1}.$$

The kernel of $\Theta$ is the subgroup
$$K=\langle \alpha_{1},\ldots,\alpha_{m}: \alpha_{1}^{2}=\cdots=\alpha_{m}^{2}=\alpha_{1} \cdots \alpha_{m}=1\rangle,$$
where $$\alpha_{j}=x^{j}yx^{1-j}, \quad j=1,\ldots,m.$$

Let us now consider the surjective homomorphism
$$\eta:K \to \langle a,b:a^{2}=b^{2}=(ab)^{2}=1\rangle$$
where
$$\eta(\alpha_{j})=\left\{\begin{array}{ll}
a, & j \equiv 1 \mod 3\\
b, & j \equiv 2 \mod 3\\
ab, & j \equiv 0 \mod 3
\end{array}
\right. \quad \mbox{if $m=3l $}
$$
$$\eta(\alpha_{j})=\left\{\begin{array}{ll}
b, & j \equiv 1 \mod 2\\
ab, & j \equiv 0 \mod 2
\end{array}
\right. \quad \mbox{if $m=4l $}
$$
and let $\Gamma$ be its kernel. Then $S={\mathbb H}^{2}/\Gamma$, $S/\langle a,b \rangle = {\mathbb H}^{2}/K$, $S/G={\mathbb H}^{2}/\Delta$ and $S/{\rm Aut}(S)={\mathbb H}^{2}/\Delta_{0}$.

\section{Proof of Theorem \ref{main}}

\subsection{}
It can be checked that the algebraic curves $S$ and the groups $G$ as described in the theorem are such that $S/G$ has signature as required. This provides the existence for the values of $m$ as desired. 

\subsubsection{}
For the curve described in (\ref{caso1}), the quotient $S/G$ has signature $(0;2,m,2m)$, where $m=2q$ and $q \geq 3$ is odd. Let us assume that ${\rm Aut}(S) \neq G.$ Then, by the lists in Singerman's paper  \cite{Singerman}, the signature of $S/{\rm Aut}(S)$ must be $(0;2,3,2m)$. Following the same article (see pp. 37), for the surjective homomorphism
$$\theta:\Delta=\langle x,y: x^{3}=y^{2m}=(xy)^{2}=1\rangle \to {\mathfrak S}_{3}$$
$$\theta(y)=(1,2), \; \theta(x)=(1,2,3), \; \theta(xy)=(1,3)$$
the group $\Gamma=\theta^{-1}(\langle (1,2) \rangle)$ is the Fuchsian group uniformizing the orbifold $S/G$ and $\Delta$ is uniformizing $S/{\rm Aut}(S)$. If $u=y$ and $v=x^{-1}yx^{-1}$, then 
$\Gamma=\langle u,v: u^{2m}=v^{2}=(uv)^{m}=1\rangle$.

If we set $x_{j}=u^{j-1}vu^{1-j}$, where $j=1,\ldots,m$, and $x_{m+1}=u^{m}$, then the subgroup $\Gamma_{0}$ generated by these elements has the presentation
$\Gamma_{0}=\langle x_{1},\ldots,x_{m+1}: x_{1}^{2}=\cdots=x_{m+1}^{2}=x_{1}x_{2}\cdots x_{m+1}=1\rangle$ and it uniformizes the orbifold $S/\langle a, b \rangle$. 

The group uniformizing $S$ is the kernel $K$ of the surjective homomorphism
$$\eta:\Gamma_{0}: \to \langle a,b \rangle$$
$$\eta(x_{2j-1})=b,\; \eta(x_{2j})=ab, \; j=1,\ldots,q,$$
$$\eta(x_{m+1})=a.$$

In order to get a contradiction, we only need to check that $K$ is not a normal subgroup of $\Delta$. If it is a normal subgroup, then, as $x_{1}x_{3} \in K$, we must have that $xx_{1}x_{3}x^{-1} \in K$. Since
$$xx_{1}x_{3}x^{-1}=yx^{-1}y^{2}x^{-1}yx^{-1}y^{-2}x^{-1}$$
and we are assuming $K$ normal in $\Delta$, we also must have that 
$$x^{-1}yx^{-1}y^{2}x^{-1}yx^{-1}y^{-2}\in K.$$

Since $x^{-1}yx^{-1}\in K$, the above asserts that $y^{2}x^{-1}yx^{-1}y^{-2}\in K$, which (again by assuming the normality) asserts that $x_{1}=x^{-1}yx^{-1}\in K$, a contradiction.

\subsubsection{}
For the curves described in (2), the signature of the quotient $S/\langle a,b,t,u\rangle$ (for any of the two cases) has signature $(0;2,4,m)$. 

If $m \geq 6$ and $m \neq 8$, then the signature $(0;2,4,m)$ is maximal \cite{Singerman}. In particular, 
$${\rm Aut}(S)=\langle a,b,t,u\rangle, \; \mbox{if $m \geq 5$ and $m \neq 8$}.$$

If $m=8$, then $S$ is the Riemann surface of genus $g=5$ described by the algebraic curve
$$\left\{\begin{array}{c}
y^{2}=x^{4}+1\\
z^{2}=x^{8}-1.
\end{array}
\right.
$$

In this case, the group $\widehat{G}=\langle a,b,t,u\rangle$ has order $8m=64$. It follows that the order of ${\rm Aut}(S)$ is of the form $64d$, some integer $d \geq 1$. By Singerman's list \cite{Singerman}, either $d=1$ (in which case, ${\rm Aut}(S)=\widehat{G}$) or $d=3$ (in which case $S/{\rm Aut}(S)$ must have signature $(0;2,3,8)$). Let us assume $d=3$.
Following Singerman's paper (see pp. 37), for the surjective homomorphism
$$\theta:\Delta=\langle x,y: x^{3}=y^{8}=(xy)^{2}=1\rangle \to {\mathfrak S}_{3}$$
$$\theta(xy)=(1,3), \; \theta(x)=(1,2,3), \; \theta(y)=(1,2)$$
the group $\Gamma=\theta^{-1}(\langle (1,2) \rangle)$ is the Fuchsian group uniformizing the orbifold $S/\widehat{G}$ and $\Delta$ is uniformizing $S/{\rm Aut}(S)$. If $u=y$ and $v=x^{-1}yx^{-1}$, then $\Gamma=\langle u,v: u^{8}=v^{2}=(uv)^{4}=1\rangle$. If we set $x_{j}=u^{j-1}vu^{1-j}$, where $j=1,\ldots,8$, then the subgroup $\Gamma_{0}$ generated by these elements has the presentation
$\Gamma_{0}=\langle x_{1},\ldots,x_{8}: x_{1}^{2}=\cdots=x_{8}^{2}=x_{1}x_{2}\cdots x_{8}=1\rangle$ and uniformizes the orbifold $S/\widehat{G}$. The derived subgroup $\Gamma_{0}'$ of $\Gamma_{0}$ uniformizes $S$. Since $v \in \Gamma_{0}$ and $u$ permutes the generators $x_{1}, \ldots, x_{8}$, we may see that $\Gamma_{0}'$ is also normal subgroup of $\Gamma$ as supossed to be. Since $y \in \Gamma$, we may see that $y$ normalizes $\Gamma_{0}'$. In our assumption ($d=3$) it must happen that $x$ also must normalize $\Gamma_{0}'$. But, $xx_{1}x^{-1}=yx$ satisfies that $\theta(yx)=(1)(2,3)$, that is, $xx_{1}x^{-1}$ does not belong to $\Gamma$, in particular, it cannot belong to $\Gamma_{0}'$; we get a contradiction.

\subsection{}
Next, we will see that the cases shown in part (\ref{caso2}) are the only situations (the case of part (\ref{caso1}) uses similar arguments and it is left to the interested reader to make the suitable modifications).

Let us assume that $S$ is a closed Riemann surface admitting a group of conformal automorphisms $G \cong {\mathbb Z}_{2}^{2} \rtimes {\mathbb Z}_{m}$ so that the quotient $S/G$ has signature $(0;2,m,m)$. Let us denote by $A$ and $B$ the generators of the normal factor ${\mathbb Z}_{2}^{2}$ and by $\widehat{T}$ the one of the cyclic factor ${\mathbb Z}_{m}$.

The quotient orbifold  ${\mathcal O}=S/\langle A,B\rangle$ has a signature of the form $(\gamma;2,\stackrel{r}{\cdots},2)$ and it admits a conformal automorphism $\widetilde{T}$ of order $m$ induced by $\widehat{T}$ which permutes the cone points. Moreover, since  ${\mathcal O}/\langle \widetilde{T} \rangle =S/G$ has signature of the form $(0;2,m,m)$, the automorphism $\widetilde{T}$ must have two fixed points and must permute the $r$ points in one orbit (so $r=m$). This in particular asserts that $\gamma=0$ and, by the Riemann-Hurwitz formula applied to the branched regular cover induced by $\langle A,B\rangle$, we obtain that  $g=m-3$. 

We may assume that  ${\mathcal O}$ is given by the Riemann sphere $\widehat{\mathbb C}$ and $\widetilde{T}$ is a M\"obius transformation of order $m$. Up to a M\"obius transformation we may assume that $\widetilde{T}(z)=\omega z$, where $\omega=e^{2 \pi i/m}$ and that the $m$ cone points are given by the $m$-roots of unity. If we now consider the  
M\"obius transformation
$$M(z)=\left(\frac{\omega+1}{\omega}\right) \frac{z-\omega}{z-1}$$
and we set 
$$\lambda_{j}=M(\omega^{2+j})=\frac{(\omega+1)(\omega^{1+j}-1)}{\omega^{2+j}-1}, \quad j=1, \ldots ,m-3,$$
we might assume that the $m$ cone points are 
$\infty, 0, 1, \lambda_{1}, \ldots , \lambda_{m-3}$
and $$\widetilde{T}(z)=\frac{(1+\omega)^{2}}{(1+\omega)^{2}-\omega z}.$$

\s

Let us now consider the following generalized Fermat curve 

$$C :  \left \{ \begin{array}{lllllll}
\,\,\,\,\,\,\,\,\,\,\, x_1^2 & + & x_2^2 & + & x_3^2 & =  & 0\\
\,\,\,\,\,\lambda_1 x_1^2 & + & x_2^2 & + & x_4^2 & =  & 0\\
\,\,\,\,\,\,\,\,\,\,\, \vdots & \; & \,\, \vdots & \; &  \,\, \vdots  &  \; & \vdots \\
\lambda_{m-3} x_1^2 & + & x_2^2 & + & x_{m}^2 & =  & 0\\\end{array} \right\} \subset {\mathbb P}^{m-1}$$
which is a closed Riemann surface of genus $g_{C}=1+2^{m-3}( m-4 )$ (for details, see \cite{CGHR} and \cite{GHL}). 

The curve $C$ admits the linear automorphisms 
$$a_{j}([x_{1}:\cdots:x_{m}])=[x_{1}:\cdots:x_{j-1}:-x_{j}:x_{j+1}:\cdots:x_{m}], \quad j=1,\ldots,m-1.$$ 

Set $a_{m}=a_1 a_2 \cdots a_{m-1}$ (multiplication by $-1$ the coordinate $x_{m}$). So, the curve $C$ admits the following abelian group of conformal automorphisms
$${\mathbb Z}_{2}^{m-1} \cong F=\langle a_1,\ldots ,a_{m-1} \rangle.$$ 

The map 
$$\pi:C \to \widehat{\mathbb C}; \hspace{0,5 cm}\pi([x_{1}:\cdots:x_{m}])=-\left({x_{2}}/{x_{1}}\right)^{2}$$
is a regular branched cover with $F$ as its deck group and whose branch values are 
$\infty, 0, 1, \lambda_{1}, \ldots, \lambda_{m-3}.$

As consequence of the results in \cite{H1}, there must be a subgroup $H \cong {\mathbb Z}_{2}^{m-3}$ of $F$ acting freely on $C$ so that $S=C/H$. 

\s

Observe that  
$$\widetilde{T}(\infty)=0, \; \widetilde{T}(0)=1, \; \widetilde{T}(1)=\lambda_1, \; \widetilde{T}(\lambda_1)=\lambda_2, \; \ldots, 
\widetilde{T}(\lambda_{m-4})=\lambda_{m-3}, \; T(\lambda_{m-3})=\infty.$$ If  $L(z)=1/z$ and $$\widetilde{U}(z)=M \circ L \circ M^{-1}(z)=-z+(1+\omega)^{2}/\omega$$ then
$$\widetilde{U}(\infty)=\infty, \; \widetilde{U}(0)=\lambda_{m-3}, \; \widetilde{U}(1)=\lambda_{m-4}, \; \widetilde{U}(\lambda_{j})=\lambda_{m-4-j},\; j=1,\ldots ,m-5.$$

\s

Note that for $m$ odd none of the values $\lambda_{j}$ is fixed by $\widetilde{U}$. If $m$ is even, then 
$\widetilde{U}$ only fixes $\lambda_{(m-4)/2}=-1$ and none of the others. Moreover, 
$\langle \widetilde{U},\widetilde{T} \rangle \cong {\mathbb D}_{m}$.

As a consequence of the results in \cite{H1}, there exist linear automorphisms
$T, U\in {\rm Aut}(C)$ (each one normalizing $F$) so that $\pi \circ T = \widetilde{T} \circ \pi$ and $\pi \circ U = \widetilde{U} \circ \pi$. In fact, by \cite{GHL}, we have that 
$$T([x_1:\cdots:x_{n+1}])=[x_{m}: \alpha_1 x_1: \alpha_2 x_2: \cdots: \alpha_{m-2} x_{m-2}: \alpha_{m-1} x_{m-1}]$$
$$\alpha_{1}=\sqrt{\lambda_{n-2}}, \; \alpha_{2}=1, \; \alpha_{j+2}=i\sqrt{\lambda_{j}}, \; j=1,\ldots ,m-3.$$

As $T$ induces $\widehat{T}$, the subgroup $H$ is normalized by $T$. Let us observe that
$$T \circ a_{j}=a_{j+1} \circ T, \quad j=1, \ldots, m-1,$$
$$T \circ a_{m}=a_{1} \circ T,$$
$$\langle T, F\rangle=F \rtimes \langle T \rangle \cong {\mathbb Z}_{2}^{m-1} \rtimes {\mathbb Z}_{m}.$$

Again from \cite{GHL}, 
$$U([x_1:\cdots:x_{m}])=
[x_1: x_{m}: i x_{m-1}: i x_{m-2}: \cdots: i x_4: i x_3:  x_2].$$

Note that
$$U^{2}([x_1:\cdots:x_{m}])=[x_1:  x_2: - x_3: - x_{4}: \cdots: - x_{m-2}: - x_{m-1}:  x_{m}] \in F$$

As a consequence of all the above, the subgroup $F$ is normal in $\langle F, T, U \rangle$ and  $$\langle F, T, U \rangle / F \cong {\mathbb D}_{m}.$$

The quotient orbifold $C/\langle F, T\rangle$ is equal to the quotient orbifold ${\mathcal O}/\langle \widetilde{T} \rangle=S/G$ whose of signature is $(0;2,m,m).$ We notice that $C/\langle F, T, U \rangle$ has signature $(0;2,4,m)$. 

\s

If $m \geq 5$ and $m \neq 8$, then the signature $(0;2,4,m)$ is maximal \cite{Singerman}; so 
$${\rm Aut}(C)=\langle F, T, U \rangle$$ and 
$C/{\rm Aut}(C)$ has signature $(0;2,4,m)$.

\s

Next we proceed to see that there exist subgroups $H$ as above only in the cases that $m$ is either divisible by $3$ or by $4$.

\begin{lemm}\label{lema1}
Let $H<F$ so that $H \cong {\mathbb Z}_{2}^{m-3}$ acts freely on $C$ and such that $T H T^{-1}=H$. Then one of the following holds.
\begin{enumerate}
\item $m \equiv 0 \mod{4}$ and 
$$H=\langle a_{1}a_{3}, a_{2}a_{4}, a_{3}a_{5},\ldots, a_{m-1}a_{1}, a_{m}a_{2} \rangle.$$

\item $m \equiv 0 \mod{3}$ and
$$H=\langle a_{1}a_{2}a_{3}, a_{2}a_{3}a_{4}, a_{3}a_{4}a_{5},\ldots, a_{m-2}a_{m-1}a_{m}, a_{m-1}a_{m}a_{1}, a_{m}a_{1}a_{2}\rangle.$$
\end{enumerate}
\end{lemm}
\begin{proof}
Let us consider a surjective homomorphism 
$\phi:F \to {\mathbb Z}_{2}^{2}$  so that 
\begin{enumerate}
\item[(i)] $H=\ker(\phi)$ does not contains the elements $a_1$,\ldots, $a_{m}$ (these are the only elements of $F$ acting with fixed points on $C;$ see \cite{GHL}); and 
\item[(ii)] $T H T^{-1}=H$.  
\end{enumerate}

If we denote by $T^*$ the automorphism of $F$ given by conjugation by $T$, then there is an automorphism 
$\rho$ of ${\mathbb Z}_{2}^{2}$  so that $$\rho \circ \phi = \phi \circ T^{*}.$$

As $a_1 \notin H$ we must have that $\phi(a_1) \neq \mbox{id}$. Set $a:=\phi(a_1)$. Now, as 
$T^{*}(a_{j})=a_{j+1}$ and $\phi$ is surjective, it should happen that $\rho(a) \neq a$. Set $b:=\rho(a)$. Then, ${\mathbb Z}_{2}^{2}=\langle a,b \rangle$.

There are only two possibilities for $\rho$; these being the following ones:

\begin{enumerate}
\item $\rho(a)=b, \; \rho(b)=a,\; \rho(ab)=ab$.
\item $\rho(a)=b, \; \rho(b)=ab, \; \rho(ab)=a$.
\end{enumerate}

In case (1) it holds that  $\phi(a_{2j-1})=a$ and $\phi(a_{2j})=b$. In this situation, we must have that  $m$ is divisible by $4$, and  
$$H=\ker(\rho)=\langle a_{1}a_{3}, a_{2}a_{4}, a_{3}a_{5},\ldots , a_{m-1}a_{1}, a_{m}a_{2} \rangle.$$

In case (2) it holds that $\phi(a_{1+3j})=a$, $\phi(a_{2+3j})=b$ and $\phi(a_{3+3j})=ab$. In this situation we must now have that 
$m$ is divisible by $3$, and 
$$H=\ker(\rho)=\langle a_{1}a_{2}a_{3}, a_{2}a_{3}a_{4}, a_{3}a_{4}a_{5},\ldots , a_{m-2}a_{m-1}a_{m}, a_{m-1}a_{m}a_{1}, a_{m}a_{1}a_{2}\rangle.$$
\end{proof}

The Riemann surface defined by $C$ is the highest abelian branched cover of the orbifold ${\mathcal O}=S/\langle A,B \rangle;$ thus it is uniquely determined up to isomorphisms.
If $\sigma \in {\rm Gal}(\overline{\mathbb Q}/{\mathbb Q})$, then $C^{\sigma}$ is also a highest abelian branched cover 
 of the orbifold ${\mathcal O};$ let us denote by $F'$ the associated deck group. By the uniqueness property, there exists an isomorphism $\varphi : C \to C^{\sigma}.$ We observe that $$F'=\langle b_1, \ldots, b_{m-1} \rangle$$where $b_i=\varphi a_i \varphi^{-1}.$  We can also consider the regular branched cover $S^{\sigma} \to \mathcal{O};$ we shall denote by $H'$ its deck group.

Lemma \ref{lema1} asserts that the only possibility for the existence of a closed Riemann surface $S$ admitting a group $G \cong {\mathbb Z}_{2}^{2} \rtimes {\mathbb Z}_{m}$, $m \geq 3$, of conformal automorphisms with $S/G$ of signature $(0;2,m,m)$ is that $m$ is either divisible by $3$ or by $4$.

\subsubsection{}
If $m$ is divisible by $4$ and not by $3$, then the normal subgroup $\langle A, B \rangle  \cong {\mathbb Z}_{2}^{2}$ has two non-trivial elements acting with fixed points (each one having exactly $m$ fixed points) whose product acts freely on $S$. Without loss of generality, we can suppose that the branch values $\mu_1=\infty, \mu_2=0, \mu_3=1, \mu_4=\lambda_1,\ldots,  \mu_m=\lambda_{m-3}$ of $\pi$ are the $m-$roots of the unity and that $\sigma$ induces the permutation $\mu_i \mapsto \mu_{i+2s}$ for some $1 \le s \le m/2-1.$ We can see that the above implies that $b_i=a_{i+2s}$ and that $$H'=\langle b_ib_{i+2}=a_{i+2s}b_{i+2+2s} : 1 \le i \le m \rangle=H.$$ It follows that $S \cong S^{\sigma}.$

\subsubsection{}
If $m$ is divisible by $3$ and not by $4$, then the three non-trivial elements of the normal subgroup $\langle A, B \rangle \cong  {\mathbb Z}_{2}^{2}$ acts with fixed points (each one having exactly $2m/3$ fixed points). Without loss of generality, we can suppose that the branch values $\mu_1=\infty, \mu_2=0, \mu_3=1, \mu_4=\lambda_1,\ldots,  \mu_m=\lambda_{m-3}$ of $\pi$ are the $m-$roots of the unity and that $\sigma$ induces the permutation $\mu_i \mapsto \mu_{i+3s}$ for some $1 \le s \le m/3-1.$ We can see that the above implies that $b_i=a_{i+3s}$ and that $$H'=\langle b_ib_{i+1}b_{i+2}=a_{i+3s}a_{i+1+3s}a_{i+2+3s} : 1 \le i \le m \rangle=H.$$It follows that $S \cong S^{\sigma}.$

In both cases $\{ \sigma : S \cong S^{\sigma}\}=\mbox{Gal}(\overline{\mathbb{Q}}/\mathbb{Q})$. Now, from a result of Wolfart on minimal field of definition of regular Belyi curves \cite{Wolfart}, we are in position to conclude that $S$ is definable over ${\mathbb Q}.$ These curves are described by the algebraic curves of parts (\ref{caso2a}) and (\ref{caso2b}) of the theorem.

\subsubsection{}
If $m$ is divisible by $3$ and $4$, then there are two possible actions. But, as observed above, in one case the three elements of order two of the normal subgroup ${\mathbb Z}_{2}^{2}$ act with fixed points and in the other case this is not the case. In particular, these two pairs are non-isomorphic and, for any of the two cases, the surface $S$ is definable over ${\mathbb Q}$.

\subsection{Isogenous decomposition of the Jacobian variety}\label{KR}
In this section we prove the isogenous decomposition of the Jacobian variety $JS$ for the quasiplatonic curves described in Theorem \ref{main}. We will use Kani-Rosen's decomposition theorem \cite{K-R}.

\begin{coro}[\cite{K-R}]\label{coroKR}
Let $S$ be a closed Riemann surface of genus $g \geq 1$ and let $H_{1},\ldots,H_{t}<{\rm Aut}(S)$ such that:
\begin{enumerate}
\item[(i)] $H_{i} H_{j}=H_{j} H_{i}$, for all $i,j =1,\ldots,t$;
\item[(ii)] $g\,(S/H_{i}H_{j})=0$, for $1 \leq i < j \leq t$
\item[(iii)] $g=\sum_{j=1}^{t} g\,(S/H_{j})$.
\end{enumerate}

Then 
$$JS \sim  \prod_{j=1}^{t} J\,(S/H_{j}).$$
\end{coro}

In each of our cases we use the following three order two cyclic groups  
$$H_{1}=\langle a \rangle, \; H_{2}=\langle b \rangle, \; H_{3}=\langle ab \rangle.$$

If $k \in \{a,b,ab\}$, we denote by $S_{k}$ the Riemann surface structure subjacent of the orbifold $S/\langle k \rangle$ and its genus by $g_{k}$.

It is clear that conditions (i) and (ii) of Corollary \ref{coroKR} are satisfied and we only need to check the condition (iii).

\subsubsection{} In the case (\ref{caso1}) of our theorem, that is, $m=2q$ ($q$ odd), we have that $S_{a}$, $S_{b}$ and $S_{ab}$ are, respectively, the following hyperelliptic curves 
$$z^{2}=x^{m}-1 \hspace{1 cm} y^{2}=x^{q}-1 \hspace{1 cm} w^{2}=x^{q}+1$$
which have respective genera equal to $g_{a}=(m-2)/2$, $g_{b}=(m-2)/4$ and $g_{ab}=(m-2)/4$; so $g_{a}+g_{b}+g_{ab}=m-2$ and condition (iii) is then satisfied. Moreover, we may see that $S_{b}$ and $S_{ab}$ are isomorphic; so we obtain that $JS \sim JS_{a} \times (JS_{b})^{2}$.

Also, let us observe that the hyperelliptic curve $S_{a}$ admits an extra conformal involution $d$, induced by $t^{q}$, with two fixed points. The quotient $S_{a,1}=S_{a}/\langle d \rangle$ has equation $w_{1}^{2}=v_{1}^{q}-1$ and $S_{a,2}=S_{a}/\langle \iota_{a} \circ d\rangle$ (where $\iota_{a}$ denotes the hyperelliptic involution), has equation $w_{2}^{2}=v_{2}(v_{2}^{q}-1)$. Clearly, $S_{a,1}$ and $S_{a,2}$ are isomorphic curves. Again, applying Kani-Rosen result, we obtain that $JS_{a} \sim (JS_{a,1})^{2}=(JS_{b})^{2}$.

\subsubsection{} In the case (\ref{caso2a}) of our theorem ($m=3l$), we have that $S_{a}$, $S_{b}$ and $S_{ab}$ are, respectively, the following hyperelliptic curves 
$$z^{2}=(x^{l}-\omega_{3})(x^{l}-\omega_{3}^{2})=x^{2l}+x^{l}+1\hspace{1 cm} y^{2}=(x^{l}-1)(x^{l}-\omega_{3}^{2}) \hspace{1 cm} w^{2}=(x^{l}-1)(x^{l}-\omega_{3})$$
of genera $g_{a}=g_{b}=g_{ab}=l-1$; so $g_{a}+g_{b}+g_{c}=m-3$ and condition (iii) is then satisfied. Moreover, we may see that $S_{a}$, $S_{b}$ and $S_{ab}$ are isomorphic, that is, $JS \sim (JS_{a})^{3}$. 

But in this case, the order $4$ automorphism $u$ of $S$ induces the automorphism $d(x,z)=(1/x,z/x^{l})$ of order two of $S_{a}$ (acting with two fixed points for $l$ odd and four fixed points if $l$ is even). We may apply Kani-Rosen's result using the groups $K_{1}=\langle d \rangle$ and $K_{2}=\langle j_{a} \circ d\rangle$, where $j_{a}(x,z)=(x,-z)$ is the hyperelliptic involution of $S_{a}$ to obtain that $JS_{a} \sim JS_{u} \times JS_{bu}$, where $S_{u}$ is the subjacent Riemann surface of the quotient 
$S_{a}/K_{1}=S/\langle u\rangle$ and $S_{bu}$ is the subjacent Riemann surface of the quotient $S_{a}/K_{2}=S/\langle bu \rangle$. To obtain explicit equations we first observe that $S_{a}$ is isomorphic to $$w^{2}=(v+1)^{2l}+(1-v^{2})^{l}+(1-v)^{2l}$$ by the isomorphism $$(x,z) \mapsto (v,w)=\left(\frac{x-1}{x+1}, z\left(\frac{2}{x+1}\right)^{l}\right)=(1-v^{2})^{l}+2\sum_{j=0}^{l}\binom{2l}{2j}v^{2j}.$$

In this new model, the automorphism $d$ is given as $d(v,w)=(-v,w)$ and $j_{a}\circ d$ is $(v,w) \mapsto (-v,-w)$. Then an equation for $S_{u}$ is given by
$$w_{1}^{2}=(1-v_{1})^{l}+2\sum_{j=0}^{l}\binom{2l}{2j}v_{1}^{j}$$
and an equation for $S_{bu}$ is given by
$$w_{2}^{2}=v_{2}\left((1-v_{2})^{l}+2\sum_{j=0}^{l}\binom{2l}{2j}v_{2}^{j}\right).$$

\subsubsection{} In the case (\ref{caso2b}) of our theorem ($m=4l$), we have that $S_{a}$, $S_{b}$ and $S_{ab}$ are, respectively, the following hyperelliptic curves 
$$z^{2}=x^{m}-1 \hspace{1 cm}y^{2}=x^{2l}-1 \hspace{1 cm}w^{2}=x^{2l}+1$$
of genera $g_{a}=2l-1$, $g_{b}=g_{ab}=l-1$; so $g_{a}+g_{b}+g_{ab}=m-3$ and condition (iii) is then satisfied. Moreover, we may see that $S_{b}$ and $S_{ab}$ are isomorphic; so $JS \sim JS_{a} \times (JS_{b})^{2}$.

Also, let us observe that the hyperelliptic curve $S_{a}$ admits an extra conformal involution $d$, induced by $t^{2l}$, with four fixed points. The quotient $S_{a,1}=S_{a}/\langle d \rangle$ has equation $w_{1}^{2}=v_{1}^{2l}-1$ and $S_{a,2}=S_{a}/\langle \iota_{a} \circ d\rangle$ (where $\iota_{a}$ denotes the hyperelliptic involution), has equation $w_{2}^{2}=v_{2}(v_{2}^{2l}-1)$. Clearly, $S_{a,1}$ and $S_{a,2}$ are not isomorphic curves as they have different genera. Applying Kani-Rosen result, we obtain that $JS_{a} \sim JS_{a,1} \times JS_{a,2}=JS_{b} \times JS_{a,2}$; so
$JS \sim (JS_{b})^{3} \times JS_{a,2}$.

\section{Curves over ${\mathbb Q}$ for the case $m$ divisible by $3$} \label{modelito}
Theorem \ref{main} asserts that if $S$ admits a group of conformal automorphisms $H \cong {\mathbb Z}_{2}^{2} \rtimes {\mathbb Z}_{m}$ so that $S/H$ has triangular signature, then $S$ is definable over ${\mathbb Q}$. In the same theorem explicit curves are provided, all of them defined over ${\mathbb Q}$ with the exception of one case: $m=3l$ and $S/H$ of triangular signature $(0;2,m,m)$. In this last case, there is provided a curve over the degree two extension ${\mathbb Q}(\omega_{3})$
$$ C: \left\{\begin{array}{l}
y^{2}=x^{2l}+\omega_{3} x^{l}+\omega_{3}^{2}\\
z^{2}=x^{2l}+x^{l}+1
\end{array}
\right.
$$

In this section we use the computation algorithm presented in \cite{HR} in order to indicate how to find another algebraic representation of $S$ defined over ${\mathbb Q}$.

Let $\Gamma={\rm Gal}({\mathbb Q}(\omega_{3})/{\mathbb Q})=\langle \sigma \rangle \cong {\mathbb Z}_{2}$, where $\sigma(\omega_{3})=\omega_{3}^{2}$. In this way the Galois orbit of $C$ consists of $C$ and the curve
$$ C^{\sigma}: \left\{\begin{array}{l}
y^{2}=x^{2l}+\omega_{3}^{2} x^{l}+\omega_{3}\\
z^{2}=x^{2l}+x^{l}+1
\end{array}
\right.
$$

The map 
$$f_{\sigma}(x,y,z)=\left(\frac{1}{x},\frac{\omega_{3}^{2} y}{x^{l}}, \frac{z}{x^{l}} \right)$$
provides an isomorphism $f_{\sigma}:C \to C^{\sigma}$. 

We may observe that $f_{\sigma}^{\sigma} \circ f_{\sigma}$ is the identity map; so the set $\{I,f_{\sigma}\}$ defines a Weil datum for $C$ with respect to the Galois extension ${\mathbb Q}(\omega_{3})/{\mathbb Q}$.

The map
$$\Phi:C \to \Phi(C) \subset {\mathbb C}^{6}: (x,y,z) \mapsto \left( x,y,z,\frac{1}{x},\frac{\omega_{3}^{2} y}{x^{l}}, \frac{z}{x^{l}} \right)$$
defines an isomorphism between $C$ and $\Phi(C)$ (its inverse is just the projection on the first three coordinates).

We consider the permutation action $\Theta:\Gamma \to {\rm GL}(6,{\mathbb C})$ given by $$\Theta(\sigma)(x_{1},x_{2},x_{3},x_{4},x_{5},x_{6})=(x_{4},x_{5},x_{6},x_{1},x_{2},x_{3}).$$

A set of generators of the algebra of $\Gamma$-invariant polynomials ${\mathbb C}[x_{1},x_{2},x_{3},x_{4},x_{5},x_{6}]^{\Gamma}$ is given by
$$t_{1}=x_{1}+x_{4}, \; t_{2}=x_{2}+x_{5}, \; t_{3}=x_{3}+x_{6},$$
$$t_{4}=x_{1}x_{4}, \; t_{5}=x_{2}x_{5}, \; t_{6}=x_{3}x_{6},$$
$$t_{7}=x_{1}x_{2}+x_{4}x_{5}, \; t_{8}=x_{1}x_{3}+x_{4}x_{6}, \; t_{9}=x_{2}x_{3}+x_{5}x_{6}.$$

\begin{rema}
Observe that if $(x_{1},\ldots,x_{6}) \in \Phi(C)$, then $t_{4}=1$.
\end{rema}

Let us consider the branched cover
$$\Psi:{\mathbb C}^{6} \to {\mathbb C}^{9}:(x_{1},\ldots,x_{6}) \mapsto (t_{1},\ldots,t_{9}).$$

The results in \cite{HR} asserts that $C$ is isomorphic to $D=\Psi(\Phi(C))$ and that $D$ is defined over ${\mathbb Q}$. In order to find the equations for $D$ we proceed as follows.

It is possible to observe that 
$$\Psi({\mathbb C}^{6})=\left\{ \begin{array}{c}
t_{1}^{2}t_{5}-t_{1}t_{2}t_{7}+t_{2}^{2}t_{4}-4 t_{4}t_{5}+t_{7}^{2}=0\\
t_{1}^{2}t_{6}-t_{1}t_{3}t_{8}+t_{3}^{2}t_{4}-4 t_{4}t_{6}+t_{8}^{2}=0\\
t_{2}^{2}t_{6}-t_{2}t_{3}t_{9}+t_{3}^{2}t_{5}-4 t_{5}t_{6}+t_{9}^{2}=0
\end{array}
\right\}.
$$

We also have the equalities
$$(1) \; y=\frac{t_{7}-t_{1}t_{2}+t_{2}x}{2x-t_{1}}$$
$$(2) \; z=\frac{t_{8}-t_{1}t_{3}+t_{3}x}{2x-t_{1}}$$
$$(3) \; x^{2}=t_{1}x-t_{4}$$

Now, (1) and (3) above assert that the equality
$$(4) \; y^{2}=x^{2l}+\omega_{3}^{2}x^{l}+\omega_{3}=(x^{l}-1)(x^{l}-\omega_{3})$$
can be written as
$$(t_{7}-t_{1}t_{2})^{2}-t_{2}^{2}t_{4}+t_{2}(2t_{7}-t_{1})x= (x^{l}-1)(x^{l}-\omega_{3})(2x-t_{1})^{2}.$$

Equality (3) asserts that 
$$(2x-t_{1})^{2}=t_{1}^{2}-4t_{4}$$
and that there are polynomials  $P,Q \in {\mathbb Q}(\omega_{3})[t_{1},t_{4}]$ so that
$$(x^{l}-1)(x^{l}-\omega_{3})=P(t_{1},t_{4})x+Q(t_{1},t_{4}).$$

\begin{rema}
If $l=2$, then $P(t_{1},t_{4})=t_{1}(t_{1}^{2}+2t_{4}-\omega_{3}^{2})$ and $Q(t_{1},t_{4})=(1+t_{4})(\omega_{3}+t_{4})-t_{1}^{2}t_{4}$.
\end{rema}

All the above asserts that (4) is equivalent to 
$$(t_{7}-t_{1}t_{2})^{2}-t_{2}^{2}t_{4}+t_{2}(2t_{7}-t_{1})x= (P(t_{1},t_{4})x+Q(t_{1},t_{4}))(t_{1}^{2}-4t_{4}),$$
from which we obtain 
$$x=R(t_{1},t_{2},t_{4},t_{7})=\frac{(t_{7}-t_{1}t_{2})^{2}-t_{2}^{2}t_{4}-(t_{1}^{2}-4t_{4})Q(t_{1},t_{4})}{(t_{1}^{2}-4t_{4})P(t_{1},t_{4})-t_{2}(2t_{7}-t_{1}) }.$$

Now, using this expression for $x$, we use (1) and (2) to obtain rational expressions for $y$ and $z$ as follows:

$$y=\frac{P(t_{1},t_{4}) ( t_{1}^3 t_{2}- t_{1}^2 t_{7}-4  t_{1} t_{2} t_{4}+4  t_{4} t_{7})+Q(t_{1},t_{4}) (t_{1}^2 t_{2}-4 t_{2} t_{4})-t_{1}^2 t_{2}^3+t_{1}^2 t_{2}^2-t_{1} t_{2} t_{7}+t_{2}^3 t_{4}+t_{2} t_{7}^2}
{P(t_{1},t_{4}) (t_{1}^3-4 t_{1} t_{4})+2 Q(t_{1},t_{4}) (t_{1}^2-4  t_{4})-2 t_{1}^2 t_{2}^2+t_{1}^2 t_{2}+2 t_{1} t_{2} t_{7}+2 t_{2}^2 t_{4}-2 t_{7}^2}$$

$$z=\frac{P(t_{1},t_{4}) (t_{1}^3 t_{3}- t_{1}^2 t_{8}-4 t_{1} t_{3} t_{4}+4 t_{4} t_{8})+Q(t_{1},t_{4}) (t_{1}^2 t_{3}-4 t_{3} t_{4})-t_{1}^2 t_{2}^2 t_{3}+t_{1}^2 t_{2} t_{3}-t_{1} t_{2} t_{8}+t_{2}^2 t_{3} t_{4}+2 t_{2} t_{7} t_{8}-t_{3} t_{7}^2}
{P(t_{1},t_{4}) (t_{1}^3-4 t_{1} t_{4})+2 Q(t_{1},t_{4}) (t_{1}^2-4 t_{4})-2 t_{1}^2 t_{2}^2+t_{1}^2 t_{2}+2 t_{1} t_{2} t_{7}+2 t_{2}^2 t_{4}-2 t_{7}^2}$$

In this way, equation 
$$y^{2}=(x^{l}-1)(x^{l}-\omega_{3})$$
can be written as an equation
$$E(t_{1},\ldots,t_{9})=0$$
and the equation
$$z^{2}=(x^{l}-\omega_{3})(x^{l}-\omega_{3}^{2})$$
can be written as an equation
$$F(t_{1},\ldots,t_{9})=0,$$
where $E,F \in {\mathbb Q}(\omega_{3})[t_{1},\ldots,t_{9}]$.

All the above asserts that $D$ is defined as the common zeroes of 
$$D: \left\{ \begin{array}{c}
t_{1}^{2}t_{5}-t_{1}t_{2}t_{7}+t_{2}^{2}t_{4}-4 t_{4}t_{5}+t_{7}^{2}=0\\
t_{1}^{2}t_{6}-t_{1}t_{3}t_{8}+t_{3}^{2}t_{4}-4 t_{4}t_{6}+t_{8}^{2}=0\\
t_{2}^{2}t_{6}-t_{2}t_{3}t_{9}+t_{3}^{2}t_{5}-4 t_{5}t_{6}+t_{9}^{2}=0\\
E(t_{1},\ldots,t_{9})=0\\
F(t_{1},\ldots,t_{9})=0
\end{array}
\right\}.
$$

The first three equations are given by polynomials with coefficients in ${\mathbb Q}$. The last two, $E$ and $F$, may still have coefficients on ${\mathbb Q}(\omega_{3})$. In this case we may change them by the following traces (which are defined over ${\mathbb Q}$ as desired), so equations for $D$ over ${\mathbb Q}$ are:
$$D: \left\{ \begin{array}{c}
t_{1}^{2}t_{5}-t_{1}t_{2}t_{7}+t_{2}^{2}t_{4}-4 t_{4}t_{5}+t_{7}^{2}=0\\
t_{1}^{2}t_{6}-t_{1}t_{3}t_{8}+t_{3}^{2}t_{4}-4 t_{4}t_{6}+t_{8}^{2}=0\\
t_{2}^{2}t_{6}-t_{2}t_{3}t_{9}+t_{3}^{2}t_{5}-4 t_{5}t_{6}+t_{9}^{2}=0\\
E(t_{1},\ldots,t_{9})+E(t_{1},\ldots,t_{9})^{\sigma}=0\\
\omega_{3}E(t_{1},\ldots,t_{9})+\omega_{3}^{2}E(t_{1},\ldots,t_{9})^{\sigma}=0\\
F(t_{1},\ldots,t_{9})+F(t_{1},\ldots,t_{9})^{\sigma}=0\\
\omega_{3}F(t_{1},\ldots,t_{9})+\omega_{3}^{2}F(t_{1},\ldots,t_{9})^{\sigma}=0
\end{array}
\right\}.
$$

\section{A remark on the decomposition of $JS$ for the case $m=6$. Group algebra point of view.}\label{GAD}
The isogenous decomposition obtained in Theorem \ref{main} for the Jacobian variety of $S$ was obtained by a simple application of  
Kani-Rosen's decomposition result (Corollary \ref{coroKR}). In this section we shall show how the methods of Lange-Recillas \cite{LR}, Carocca-Rodr\'iguez \cite{CR},  Rojas \cite{RJ} and Jim\'enez \cite{yo} about decompositions of abelian varieties, using the rational algebra of finite groups, can also be applied to obtain the same decomposition of Theorem \ref{main}. We only describe it for $m=6$ in case $\eqref{caso2a}$ as the general case follows the same ideas.

\subsection{The group algebra decomposition for abelian varieties with non-trivial automorphisms}
We start by recalling some definitions and results about the isotypical decomposition of any abelian variety $A$ with a non-trivial (finite) group of automorphisms in terms of the complex and rational irreducible representations of $G$.

Let $V$ be an irreducible representation of $G$ over ${\mathbb C}$. If  we denote by $F$ its field of definition and by $K$ the field obtained by extending ${\mathbb Q}$ by the values of the character $\chi_V$, then $F$ is a finite extension of $K$ and the extension degree $m_V=[F:K]$ is called the Schur index of $V$. For details see \cite{Serre}.

The action of $G$ on $A$ induced a $\mathbb{Q}-$algebra homomorphism $\rho:{\mathbb Q}[G]\to \text{End}_{\mathbb Q}(A)$. For any element $\alpha \in {\mathbb Q}[G]$ we define an abelian subvariety $B_{\alpha} := {\rm Im} (\alpha)=\rho(l\alpha)(A) \subset A$, where $l$ is some positive integer such that $l\alpha \in {\mathbb Z}[G]$. 

The semi-simple algebra ${\mathbb Q}[G]$ decomposes into a product $Q_0 \times \dots \times Q_r$ of simple ${\mathbb Q}-$algebras; the simple algebras $Q_i$ are in bijective correspondence with the rational irreducible representations of $G$. That is, for any rational irreducible representation ${W}_i$ of $G$ there is a uniquely determined central idempotent $e_i$. This idempotent defines an abelian subvariety of $A$, namely $B_i=B_{e_i}$. These varieties, called isotypical components, are uniquely determined by the representation ${W}_i$. Moreover, the decomposition of every $Q_i=L_1\times \dots \times L_{n_i}$ into a product of minimal left ideals (all isomorphic) gives a further decomposition of $A$. More precisely, there are idempotents $f_{i1},\dots, f_{in_i}\in Q_i$ such that $e_i=f_{i1}+\dots +f_{in_i}$ where $n_i=\text{dim} V_i/m_{V_i}$, with $V_i$ the complex irreducible representation associated to ${W}_i$. These idempotents provide subvarieties $B_{ij}:=B_{f_{ij}} \sim B_i$, for all $j$. Then we have the following theorem

\begin{theo} [\cite{LR}, \cite{CR}] \label{teolangerecillas}
Let $G$ be a finite group acting on an abelian variety $A.$ Let $W_1, \ldots, W_r$ denote the irreducible rational representations of $G$. Then there are abelian subvarieties $B_1, \ldots, B_r$ of $A$ and an isogeny 
\begin{equation} \label{lg}
A \sim B_1^{n_1} \times \cdots \times B_r^{n_r}.
\end{equation}
\end{theo}

\s

The above isogenous decomposition of the abelian variety $A$ is called {\it the group algebra decomposition of $A$}. 

\s
\subsection{The group algebra decomposition for the Jacobian variety of Riemann surfaces}
We next assume that $A=JS$, where $S$ is a closed Riemann surface and $G$ is a (finite) group of conformal automorphisms of it.
In this particular case, in Theorem \ref{teolangerecillas} we always have that one of the factors $B_j$ is isogenous to the Jacobian variety $JS_{G}$, where $S_{G}$ is the subjacent Riemann surface structure associated to the Riemann orbifold $S/G$. In this situation, the isotypical decomposition can be made more explicitly as follows.

Let $H$ be a subgroup of $G.$ We denote by $\pi_H : S \to S_H=S/H$ the associated regular covering map and by $\rho_H$ the representation of $G$ induced by the trivial representation of $H$. If $U$ and $V$ are representations of $G$, then $\langle U,V \rangle$ denotes the usual inner product of the corresponding characters. By the Frobenius Reciprocity Theorem $\langle \rho_H,V\rangle=\text{dim}_{\mathbb C} V^H$, where $V^H$ is the subspace of $V$ fixed by $H$.  Define $p_H=\frac{1}{\vert H\vert}\sum_{h\in H} h$ as the central idempotent in ${\mathbb Q}[H]$; corresponding to the trivial representation of $H$. Also, we define $f_H^{i}$ as $p_H e_{i}$, an idempotent element in ${\mathbb Q}[G]e_{i}$.

With the previous notations, the corresponding group algebra decomposition of $JS_H$ is given as follows \cite[Proposition 5.2]{CR}:
\begin{equation}\label{carocca-rodriguez}
J S_H \sim J S_G\times B_{1}^{\frac{\text{dim} V_1^H}{m_1}}\times \dots \times B_{r}^{\frac{\text{dim} V_r^H}{m_r}},
\end{equation}
whit $m_i=m_{V_i}$. Moreover, 
\begin{equation}\label{phjxh}
{\rm Im} (p_H)=\pi_H^*(J_H)
\end{equation}
where $\pi_H^*(J S_H)$ is the pullback of $J S_H$ by $\pi_H$. If $\text{dim} V_i^H\neq 0$ then
\begin{equation}\label{fh}
{\rm Im}(f_H^i)=B_{i}^{\frac{\text{dim} V_i^H}{m_{i}}}.
\end{equation}

\s

We should notice that the previous results do not depend on the action of $G.$ The next result related to the dimension of the factors in (\ref{lg}) involves the way the group $G$ acts.

\begin{theo}{\cite{RJ}}
Let $G$ be a finite group acting on a compact Riemann surface $S$ with geometric signature given by $(\gamma; [m_1,C_1],\ldots,[m_r,C_r])$. Then the dimension of factor $B_i$ associated to a non trivial rational irreducible representation $W_i$ in (\ref{lg}) is given by

$$\dim B_i=k_i(\dim V_i (\gamma-1)+\frac{1}{2} \sum_{k=1}^r ( \dim V_i -\dim (V_i^{G_k}))$$where $G_k$ is a representative of the conjugacy class of $C_k$ and $k_i = m_i |\mbox{Gal}(K_i/\mathbb{Q}) |$.
\end{theo}

\s

The following lemma gives us conditions under which a factor in the group algebra decomposition can be described as the image of a concrete idempotent, in particular, when it corresponds to a Jacobian of an intermediate quotient.

\begin{lemm}\cite{yo}\label{metodo}
Let $S$ be a Riemann surface with an action of a finite group $G$ such that the genus of $S/G$ is equal to zero.  Assume that $V_{1}$, \ldots, $V_{q}$ are the non-isomorphic complex irreducible representations of $G$. Let us consider the group algebra decomposition of $JS$ given by (\ref{lg}). Let $H$ be a subgroup of $G$ so that $\dim_{\mathbb C} V_i^{H}=m_i$, for some fixed index $i$. Then
\begin{itemize}
\item[(i)]  ${\rm Im} (f_H^i)=B_{i}$;
\item[(ii)] if, moreover, $\dim_{\mathbb C} V_l^{H} = 0$ for all $l$, $l\neq i$, 
such that $\dim_{\mathbb C} B_l\neq 0$ then 
$$J S_H \sim {\rm Im} (p_H)=B_{i}.$$
\end{itemize}
\end{lemm}

\s

On this way, in order to obtain factors isogenous to Jacobian varieties at the group algebra decomposition of $JS$, we need to look for subgroups $H$ of $G$ satisfying $\dim_{\mathbb C} V_i^{H}=\langle \rho_H ,V_i\rangle=m_i$ and $\dim_{\mathbb C} V_l^{H} = 0$ for all $l$, $l\neq i$, such that $\dim_{\mathbb C} B_l\neq 0$.

\subsection{Our examples}\label{iso}
In our case, the group $$G=\langle a,t: a^{2}=t^{6}=[a,t]^{2}=1, t^{3}=(at)^{3}\rangle$$ has eight complex irreducible representations $V_1, \ldots, V_8$, as shown in the following character table.
\s

\begin{center}
\begin{tabular}{ | c |  c | c |c |c |c |c |c | c | }
\hline Conj. class     & $\mbox{id}$ & $at^3$ & $a$ & $t^3$& $t^2$& $t^5$& $t^4$& $t$\\
\hline $V_1$ & 1 & 1 & 1& 1& 1& 1& 1& 1 \\
\hline $V_2$ & 1 & -1 & 1& -1& 1& -1& 1& -1 \\
\hline $V_3$ & 1 & -1 & 1& -1& $\xi^2$ & $-\xi^2$ & $\xi$ & $-\xi$ \\
\hline $V_4$ & 1 & -1 & 1& -1& $\xi$ & $-\xi$ & $\xi^2$ & $-\xi^2$ \\
\hline $V_5$ & 1 & 1 & 1& 1& $\xi^2$ & $\xi^2$ & $\xi$ & $\xi$ \\
\hline $V_6$ & 1 & 1 & 1& 1& $\xi$ & $\xi$ & $\xi^2$ & $\xi^2$ \\
\hline $V_7$ & 3 & 1 & -1& -3& 0& 0& 0& 0 \\
\hline $V_8$ & 3 & -1 & -1& 3& 0& 0& 0& 0 \\
\hline
\end{tabular}
\end{center}where $\xi = \mbox{exp}(2 \pi i / 3).$ 

It is not difficult to see that the rational irreducible representations of $G$ are $$W_1:=V_1, W_2:=V_2, W_3:=V_3 \oplus V_4, W_4:=V_5 \oplus V_6, W_5:= V_7, W_6:=V_8. $$

By applying Theorem \ref{teolangerecillas} we obtain 
$$JS \sim B_1^1 \times B_2^1 \times B_3^1 \times B_4^1 \times B_5^{3} \times B_6^{3}.$$ 

Moreover, as $B_1 \sim JS_G$ (and $S/G$ has genus zero), $B_1=0.$ Finally, with the help of a computational program such as MAGMA \cite{magma} we can obtain that $\mbox{dim}(B_2)=\mbox{dim}(B_3)=\mbox{dim}(B_4)=\mbox{dim}(B_6)=0.$ Combining this fact with the previous isogenies, we are in position to conclude that $JS \sim B_5^3.$

Then now we are looking for every (conjugacy class of) subgroup $H$ of $G$ satisfying $\dim_{\mathbb C} V_7^{H}=\langle \rho_H,V_7\rangle=1$. Remember that $W_5=V_7$. Hence, again using MAGMA we obtain the table of induced representation by any $H\subseteq G$. Therefore, the class of subgroups $H$ satisfying this is which given by $H=\langle a\rangle$. Thus 
$$JS\sim B_5^3\sim (JS_{\langle a\rangle})^3.$$
We obtain that $\mbox{dim}(B_5)=1;$ thus $JS$ in this case is completely decomposable (see also Section \ref{isogenia}).


\bigskip

\end{document}